\documentclass[reqno]{amsart}

\usepackage{amssymb,enumitem,mathrsfs,mleftright,soul,tensor}
\usepackage[colorlinks]{hyperref}
\usepackage[capitalize,nameinlink]{cleveref}

\newtheorem{Lem}{Lemma}[section]
\newtheorem{Thm}[Lem]{Theorem}

\theoremstyle{definition}
\newtheorem{Eg}[Lem]{Example}
\newtheorem{Def}[Lem]{Definition}
\newtheorem{Rmk}[Lem]{Remark}

\setlist[enumerate]{leftmargin = *}
\setlist[itemize]{leftmargin = *}

\numberwithin{equation}{section}

\newcommand{\A}{\mathscr{A}}
\newcommand{\B}[1]{\func{\mathcal{B}}{#1}}
\newcommand{\C}{\mathbb{C}}
\newcommand{\D}{\displaystyle}
\newcommand{\e}{\operatorname{e}}
\newcommand{\E}[3]{\func{E}{#1,#2;#3}}
\newcommand{\G}{\mathcal{G}}
\newcommand{\I}[1]{\func{\mathcal{I}}{#1}}
\newcommand{\K}[1]{\func{\mathbin{\mathbb{K}}}{#1}}
\newcommand{\M}[1]{\func{M}{#1}}

\newcommand{\X}[3]{\func{\mathsf{X}}{#1,#2;#3}}
\newcommand{\Y}{\mathsf{Y}}
\newcommand{\Z}[1]{\func{\mathcal{Z}}{#1}}
\newcommand{\Br}[1]{\( #1 \)}
\newcommand{\Cb}[1]{\func{C_{b}}{#1}}
\newcommand{\Cc}[1]{\func{C_{c}}{#1}}
\newcommand{\CC}[2]{\SqBr{#1,#2}}
\newcommand{\Cl}[2]{\overline{#1}^{#2}}
\newcommand{\Co}[1]{\func{C_{0}}{#1}}
\newcommand{\df}{\stackrel{\operatorname{df}}{=}}
\newcommand{\Do}[3]{\func{D_{0}}{#1,#2;#3}}
\newcommand{\Eo}[3]{\func{E_{0}}{#1,#2;#3}}

\newcommand{\GG}{\mathscr{G}}
\newcommand{\id}{\operatorname{id}}

\newcommand{\lt}{\operatorname{lt}}
\newcommand{\Xo}[3]{\func{\mathsf{X}_{0}}{#1,#2;#3}}

\newcommand{\Aut}[1]{\func{\operatorname{Aut}}{#1}}
\newcommand{\FCP}[2]{\func{C^{\ast}}{#1,#2}}
\newcommand{\Fix}[3]{\func{\operatorname{Fix}}{#1,#2;#3}}
\newcommand{\Int}[4]{\int_{#1} #2 ~ \mathrm{d}{\func{#3}{#4}}}

\newcommand{\Map}[4]{\mleft\{ \begin{matrix} #1 & \to & #2 \\ #3 & \mapsto & \D #4 \end{matrix} \mright\}}
\newcommand{\RCP}[2]{\func{C^{\ast}_{\r}}{#1,#2}}
\newcommand{\Seq}[2]{\Br{#1}_{#2}}
\newcommand{\Set}[2]{\mleft\{ #1 ~ \middle| ~ #2 \mright\}}

\newcommand{\FixR}[4]{\func{\operatorname{Fix}_{\operatorname{Rieffel}}}{#1,#2,#3;#4}}
\newcommand{\func}[2]{#1 \Br{#2}}
\newcommand{\Func}[2]{\func{\Br{#1}}{#2}}
\newcommand{\FUNC}[2]{\func{\SqBr{#1}}{#2}}
\newcommand{\Norm}[1]{\mleft\| #1 \mright\|}
\newcommand{\Pair}[2]{\Br{#1,#2}}
\newcommand{\PCDS}[4]{\Br{#1,#2,#3;#4}}
\newcommand{\PGDS}[3]{\Br{#1,#2;#3}}
\newcommand{\Prim}[1]{\func{\operatorname{Prim}}{#1}}
\newcommand{\Quad}[4]{\Br{#1,#2,#3,#4}}
\newcommand{\Quot}[2]{#1 / #2}
\newcommand{\Span}[1]{\func{\operatorname{Span}}{#1}}
\newcommand{\SqBr}[1]{\[ #1 \]}
\newcommand{\SSet}[1]{\mleft\{ #1 \mright\}}

\newcommand{\Trip}[3]{\Br{#1,#2,#3}}
\newcommand{\Unit}[1]{#1^{\Br{0}}}

\newcommand{\Gammb}[1]{\func{\Gamma_{b}}{#1}}
\newcommand{\Gammc}[1]{\func{\Gamma_{c}}{#1}}
\newcommand{\Gammo}[1]{\func{\Gamma_{0}}{#1}}
\newcommand{\Inner}[3]{\mleft\langle #1 \middle| #2 \mright\rangle_{#3}}
\newcommand{\InvIm}[2]{#1^{- 1} \SqBr{#2}}
\newcommand{\Orbit}[2]{#1 \backslash #2}
\newcommand{\alphArg}[2]{\func{\alpha_{#1}}{#2}}

\newcommand{\alphExtArg}[2]{\func{\overline{\alpha}_{#1}}{#2}}

\renewcommand{\(}{\mleft(}
\renewcommand{\)}{\mright)}
\renewcommand{\[}{\mleft[}
\renewcommand{\]}{\mright]}

\renewcommand{\r}{\operatorname{r}}
\renewcommand{\u}{\operatorname{u}}
\renewcommand{\Im}[2]{#1 \SqBr{#2}}
\renewcommand{\ker}[1]{\func{\operatorname{ker}}{#1}}

\allowdisplaybreaks

\begin{document}

\title[Bundles of Generalized Fixed-Point Algebras]{Bundles of Generalized Fixed-Point Algebras for Proper Groupoid Dynamical Systems}
\author{Jonathan H. Brown and Leonard T. Huang}
\address{Jonathan Brown, Department of Mathematics, University of Dayton, Dayton, Ohio 45469, U.S.A.}
\email{jbrown10@udayton.edu}
\address{Leonard Huang, Department of Mathematics, University of Colorado Boulder, Boulder, Colorado 80309, U.S.A.}
\email{Leonard.Huang@Colorado.EDU}

\begin{abstract}
In this paper, we show that the generalized fixed-point algebra of a proper groupoid dynamical system, under certain assumptions, may be fibered over any locally compact Hausdorff space to which a continuous map exists from the unit space of the underlying groupoid. We will also provide some important examples.
\end{abstract}



\maketitle


\section*{Introduction}


In 1990, Marc Rieffel defined proper $ C^{\ast} $-dynamical systems to extend the idea of proper topological dynamical systems to noncommutative spaces (\cite{Rieffel2}). In 2009, the first author extended Rieffel's work in his PhD thesis to incorporate groupoid actions by defining proper groupoid dynamical systems. The results of this thesis later appeared in \cite{Brown}.

The main result of \cite{Rieffel2} is that a proper $ C^{\ast} $-dynamical system gives rise to a $ C^{\ast} $-algebra, called the \emph{generalized fixed-point algebra}, that is Morita equivalent to an ideal of the reduced crossed product of the system. It was shown in \cite{Brown} that generalized fixed-point algebras can be defined analogously for proper groupoid dynamical systems.

Rieffel also showed in \cite{Rieffel2} that if a proper $ C^{\ast} $-dynamical system associated to a constant group bundle over a locally compact Hausdorff space can be fibered, in a certain sense, into proper $ C^{\ast} $-dynamical systems, then the generalized fixed-point algebra of the original proper $ C^{\ast} $-dynamical system can likewise be fibered over the space into generalized fixed-point algebras. It is our goal in this paper to offer a far-reaching generalization of this result.


\section{Preliminaries}


This paper concerns $ \Co{X} $-algebras, for a locally compact Hausdorff space $ X $. A \emph{$ \Co{X} $-algebra} is a pair $ \Pair{A}{\Phi} $ such that
\begin{itemize}
\item
$ A $ is a $ C^{\ast} $-algebra and

\item
$ \Phi: \Co{X} \to \Z{\M{A}} $ is a non-degenerate $ \ast $-homomorphism.
\end{itemize}
Note that $ \Phi $ is non-degenerate if and only if for any approximate identity $ \Seq{e_{i}}{i \in I} $ in $ A $, we have $ \D \lim_{i \in I} \FUNC{\func{\Phi}{e_{i}}}{a} = a $ for all $ a \in A $. As we will consider $ C^{\ast} $-algebras $ A $ with multiple $ \Co{X} $-algebra structures, we will not, as is usual, drop $ \Phi $ from the notation. We will denote by $ \func{J}{A,\Phi;x} $ the ideal
$$
\Cl{\Span{\Set{\FUNC{\func{\Phi}{\varphi}}{a}}{\varphi \in \Co{X}, ~ \func{\varphi}{x} = 0, ~ \text{and} ~ a \in A}}}{A}
$$
of $ A $, and $ \Pair{A}{\Phi}_{x} $ the quotient $ C^{\ast} $-algebra $ \Quot{A}{\func{J}{A,\Phi;x}} $.

For every $ \Co{X} $-algebra $ \Pair{A}{\Phi} $, there is a unique upper-semicontinuous $ C^{\ast} $-bundle $ \Pair{\A^{A,\Phi}}{p^{A,\Phi}} $, where
\begin{itemize}
\item
$ \A^{A,\Phi} $ is a topological space, with underlying set $ \D \bigsqcup_{x \in X} \Pair{A}{\Phi}_{x} $, satisfying properties that are listed in \cite[Definition C.16]{Williams}.

\item
$ p^{A,\Phi}: \A^{A,\Phi} \to X $ is a continuous open surjection, and $ \func{p^{A,\Phi}}{x,a} = x $ for all $ x \in X $ and $ a \in \Pair{A}{\Phi}_{x} $.
\end{itemize}
For more details about $ \Co{X} $-algebras and upper-semicontinuous $ C^{\ast} $-bundles, see \cite[Appendix C]{Williams}.

The next result says that ideals of $ \Co{X} $-algebras are also $ \Co{X} $-algebras.



\begin{Lem} \label{An Ideal of a C0(X)-Algebra Is a C0(X)-Algebra}
Let $ \Pair{A}{\Phi} $ be a $ \Co{X} $-algebra, and $ E $ an ideal of $ A $. Then $ \Pair{E}{\Phi_{E}} $ is also a $ \Co{X} $-algebra, where $ \Phi_{E}: \Co{X} \to \Z{\M{E}} $ is defined by
$$
\forall \varphi \in \Co{X}: \quad
\func{\Phi_{E}}{\varphi} \df \func{\Phi}{\varphi}|_{E}.
$$
\end{Lem}


\begin{proof}
As $ E $ is an ideal of $ A $, we have $ \func{\Phi_{E}}{\varphi} \in \Z{\M{E}} $ for all $ \varphi \in \Co{X} $. It remains to see that $ \Phi_{E} $ is non-degenerate. Let $ \Seq{\varphi_{i}}{i \in I} $ be an approximate identity of $ \Co{X} $. As $ \Phi $ is non-degenerate, we have
$$
\forall a \in E: \quad
  \lim_{i \in I} \FUNC{\func{\Phi_{E}}{\varphi_{i}}}{a}
= \lim_{i \in I} \FUNC{\func{\Phi}{\varphi_{i}}}{a}
= b.
$$
This completes the proof.
\end{proof}



\begin{Rmk}
\cref{An Ideal of a C0(X)-Algebra Is a C0(X)-Algebra} is a generalization of \cite[Proposition 3.3]{Rieffel2}.
\end{Rmk}


We will also need the following lemma, which describes how to transfer $ \Co{X} $-algebra structures across imprimitivity bimodules. See \cite{Raeburn/Williams} for more details about imprimitivity bimodules.



\begin{Lem} \label{Transferring C0(X)-Algebra Structures Across Imprimitivity Bimodules}
Let $ A,B $ be $ C^{\ast} $-algebras, and $ \Y $ an imprimitivity $ \Pair{A}{B} $-bimodule. Suppose that $ \Pair{B}{\Psi} $ is a $ \Co{X} $-algebra. Then $ \Pair{A}{\Phi} $ is a $\Co{X}$-algebra, where $\Phi$ is characterized by
$$
\forall \varphi \in \Co{X}, ~ \forall \zeta,\eta \in \Y, ~ \forall b \in B: \quad
\FUNC{\func{\Phi}{\varphi}}{\Inner{\zeta}{\eta \bullet b}{\Y,A}} = \Inner{\zeta}{\eta \bullet \FUNC{\func{\Psi}{\varphi}}{b}}{\Y,A}.
$$
Furthermore, $ \Pair{A}{\Phi} $ satisfies
$$
\forall x \in X: \quad
\Pair{A}{\Phi}_{x} = \Quot{A}{\func{h_{\Y}}{\func{J}{B,\Psi;x}}},
$$
where $ h_{\Y} $ denotes the Rieffel correspondence map for $ \Y $ (see \cite[Proposition 3.24]{Raeburn/Williams}).
\end{Lem}


\begin{proof}
Let $ \I{A} $ and $ \I{B} $ denote, respectively, the set of ideals of $ A $ and $ B $. Let $ \Prim{A} $ and $ \Prim{B} $ denote, respectively, the set of primitive ideals of $ A $ and $ B $. Recall that $ h_{\Y} $ is a lattice-structure-preserving bijection from $ \I{B} $ to $ \I{A} $, and that the restriction of $ h_{\Y} $ to $ \Prim{B} $ is a homeomorphism from $ \Prim{B} $ to $ \Prim{A} $, if both sets are equipped with the Jacobson topology.

Let $ \Omega_{A}: \Cb{\Prim{A}} \to \Z{\M{A}} $ and $ \Omega_{B}: \Cb{\Prim{B}} \to \Z{\M{B}} $ denote, respectively, the $ \ast $-isomorphisms for $ A $ and $ B $ coming from the Dauns-Hofmann Theorem. As $ \Pair{B}{\Psi} $ is a $ \Co{X} $-algebra, \cite[Proposition C.5]{Williams} tells us that there exists a continuous map $ \sigma: \Prim{B} \to X $ such that
$$
\forall \varphi \in \Co{X}: \quad
\func{\Psi}{\varphi} = \func{\Omega_{B}}{\varphi \circ \sigma}.
$$
Define a continuous map $ \tau: \Prim{A} \to X $ by
$$
\forall P \in \Prim{A}: \quad
\func{\tau}{P} \df \func{\sigma}{\func{h_{\Y}^{- 1}}{P}}.
$$
Define $ \Phi: \Co{X} \to \Z{\M{A}} $ by
$$
\forall \varphi \in \Co{X}: \quad
\func{\Phi}{\varphi} = \func{\Omega_{A}}{\varphi \circ \tau}.
$$
Then by \cite[Proposition C.5]{Williams} again, $ \Pair{A}{\Phi} $ is a $ \Co{X} $-algebra.

Now, we have for every $ \varphi \in \Co{X} $, $ \zeta,\eta \in \Y $, $ b \in B $, and $ P \in \Prim{A} $ that
\begin{align*}
        \FUNC{\func{\Phi}{\varphi}}{\Inner{\zeta}{\eta \bullet b}{\Y,A}} + P
=   & ~ \FUNC{\func{\Omega_{A}}{\varphi \circ \tau}}{\Inner{\zeta}{\eta \bullet b}{\Y,A}} + P \\
=   & ~ \func{\varphi}{\func{\tau}{P}} \cdot \Inner{\zeta}{\eta \bullet b}{\Y,A} + P \\
    & ~ \Br{\text{By the Dauns-Hofmann Theorem.}} \\
=   & ~ \Inner{\zeta}{\func{\varphi}{\func{\tau}{P}} \cdot \Br{\eta \bullet b}}{\Y,A} + P \\
=   & ~ \Inner{\zeta}{\eta \bullet \SqBr{\func{\varphi}{\func{\tau}{P}} \cdot b}}{\Y,A} + P \\
=   & ~ \Inner{\zeta}{\eta \bullet \SqBr{\func{\varphi}{\func{\tau}{P}} \cdot b + Q}}{\Y,A} + P \\
    & ~ \Br{\text{Apply Rieffel correspondence to $ Q \df \func{h_{\Y}^{- 1}}{P} $.}} \\
=   & ~ \Inner{\zeta}{\eta \bullet \SqBr{\func{\varphi}{\func{\sigma}{\func{h_{\Y}^{- 1}}{P}}} \cdot b + Q}}{\Y,A} + P \\
=   & ~ \Inner{\zeta}{\eta \bullet \SqBr{\func{\varphi}{\func{\sigma}{Q}} \cdot b + Q}}{\Y,A} + P \\
=   & ~ \Inner{\zeta}{\eta \bullet \SqBr{\Func{\varphi \circ \sigma}{Q} \cdot b + Q}}{\Y,A} + P \\
=   & ~ \Inner{\zeta}{\eta \bullet \SqBr{\FUNC{\func{\Omega_{B}}{\varphi \circ \sigma}}{b} + Q}}{\Y,A} + P \\
    & ~ \Br{\text{By the Dauns-Hofmann Theorem again.}} \\
=   & ~ \Inner{\zeta}{\eta \bullet \SqBr{\FUNC{\func{\Psi}{\varphi}}{b} + Q}}{\Y,A} + P \\
=   & ~ \Inner{\zeta}{\eta \bullet \FUNC{\func{\Psi}{\varphi}}{b}}{\Y,A} + P. \\
    & ~ \Br{\text{By Rieffel correspondence again.}}
\end{align*}
As $ \D \bigcap_{P \in \Prim{A}} = \SSet{0_{A}} $, we find that
$$
\forall \varphi \in \Co{X}, ~ \forall \zeta,\eta \in \Y, ~ \forall b \in B: \quad
\FUNC{\func{\Phi}{\varphi}}{\Inner{\zeta}{\eta \bullet b}{\Y,A}} = \Inner{\zeta}{\eta \bullet \FUNC{\func{\Psi}{\varphi}}{b}}{\Y,A}.
$$
As $ \Y $ is a full Hilbert $ \Pair{A}{B} $-bimodule, it follows that $ \func{J}{A,\Phi;x} = \func{h_{\Y}}{\func{J}{B,\Psi;x}} $ for every $ x \in X $, so
$ \Pair{A}{\Phi}_{x} = \Quot{A}{\func{h_{\Y}}{\func{J}{B,\Psi;x}}} $ as required.
\end{proof}


For the rest of this section, we will assume that $ \G $ is a second-countable locally compact Hausdorff groupoid, $ \lambda $ is a Haar system on $ \G $, and $ \alpha $ is an action of $ \G $ on a separable $ \Co{\Unit{G}} $-algebra $ A $. Hence, $ \GG \df \Quad{\G}{A}{\Phi}{\alpha} $ is a separable $ C^{\ast} $-dynamical system. For each $ u,v \in \Unit{\G} $, let
\begin{align*}
\G_{u}     & \df \Set{\gamma \in \G}{\func{s_{\G}}{\gamma} = u}, \\
\G^{v}     & \df \Set{\gamma \in \G}{\func{r_{\G}}{\gamma} = v}, \\
\G_{u}^{v} & \df \G_{u} \cap \G^{v}.
\end{align*}
Let $ \func{r_{\G}^{\ast}}{\A^{A,\Phi},p^{A,\Phi}} $ denote the pullback of $ \Pair{\A^{A,\Phi}}{p^{A,\Phi}} $ under $ r_{\G} $, i.e.,
$$
    \func{r_{\G}^{\ast}}{\A^{A,\Phi},p^{A,\Phi}}
\df \Pair{\Set{\Pair{\gamma}{a} \in \G \times \A^{A,\Phi}}{\func{r_{\G}}{\gamma} = \func{p^{A,\Phi}}{a}}}{p^{A,\Phi} \circ r_{\G}}.
$$
Then $ \func{r_{\G}^{\ast}}{\A^{A,\Phi},p^{A,\Phi}} $ is an upper-semicontinuous $ C^{\ast} $-bundle over $ \G $, and we will make the identification
$$
\Gammc{\func{r_{\G}^{\ast}}{\A^{A,\Phi},p_{A,\Phi}}}
~ \leftrightarrow ~
\Set{f \in \Cc{\G,\A^{A,\Phi}}}{\func{f}{\gamma} \in \Pair{A}{\Phi}_{\func{r_{\G}}{\gamma}} ~ \text{for all} ~ \gamma \in \G}.
$$



\begin{Def}[\cite{Goehle,Khoshkam/Skandalis}]
The \emph{full groupoid crossed product} of $ \Pair{\GG}{\lambda} $, denoted by $ \FCP{\GG}{\lambda} $, is defined as the universal enveloping $ C^{\ast} $-algebra of the normed convolution $ \ast $-algebra
$$
\Quad{\Gammc{\func{r_{\G}^{\ast}}{\A^{A,\Phi},p^{A,\Phi}}}}{\star_{\GG,\lambda}}{^{\ast_{\GG,\lambda}}}{\Norm{\cdot}_{I}^{\GG,\lambda}},
$$
and we let $ \pi_{\u}^{\GG,\lambda} $ denote the canonical injective $ \ast $-homomorphism
$$
\Trip{\Gammc{\func{r_{\G}^{\ast}}{\A^{A,\Phi},p^{A,\Phi}}}}{\star_{\GG,\lambda}}{^{\ast_{\GG,\lambda}}} \hookrightarrow \FCP{\GG}{\lambda}.
$$

The \emph{reduced groupoid crossed product} of $ \Pair{\GG}{\lambda} $, denoted by $ \RCP{\GG}{\lambda} $, is defined as the closure of $ \Trip{\Gammc{\func{r_{\G}^{\ast}}{\A^{A,\Phi},p^{A,\Phi}}}}{\star_{\GG,\lambda}}{^{\ast_{\GG,\lambda}}} $ with respect to the norm $ \Norm{\cdot}_{\r}^{\GG,\lambda} $ induced from regular representations.  We let $ \pi_{\r}^{\GG,\lambda} $ denote the canonical injective $ \ast $-homomorphism
$$
\Trip{\Gammc{\func{r_{\G}^{\ast}}{\A^{A,\Phi},p^{A,\Phi}}}}{\star_{\GG,\lambda}}{^{\ast_{\GG,\lambda}}} \hookrightarrow \RCP{\GG}{\lambda}.
$$
\end{Def}


For more information on groupoid crossed products, please see the excellent exposition in \cite{Muhly/Williams}.



\begin{Def}[\cite{Brown}] \label{Proper Groupoid Dynamical System}
A \emph{proper groupoid dynamical system} is a triple $ \PGDS{\GG}{\lambda}{A_{0}} $ with the following properties:
\begin{enumerate}
\item
$ \GG = \Quad{\G}{A}{\Phi}{\alpha} $ is a (separable) groupoid dynamical system.

\item
$ \lambda $ is a Haar system on $ \G $.

\item
$ A_{0} $ is a dense $ \ast $-subalgebra of $ A $.

\item
For each $ a,b \in A_{0} $, the continuous map
$$
\Map{\G}{\A^{A,\Phi}}{\gamma}{\func{a}{\func{r_{\G}}{\gamma}}^{\ast} \alphArg{\gamma}{\func{b}{\func{s_{\G}}{\gamma}}}}
$$
is such that for any net $ \Seq{\varphi_{i}}{i \in I} $ in $ \Cc{\G,\CC{0}{1}} $ converging uniformly to $ 1 $ on compact subsets of $ \G $, the corresponding net
$$
\Seq{\Map{\G}{\A^{A,\Phi}}{\gamma}{\func{\varphi_{i}}{\gamma} \cdot \func{a}{\func{r_{\G}}{\gamma}}^{\ast} \alphArg{\gamma}{\func{b}{\func{s_{\G}}{\gamma}}}}}{i \in I}
$$
in $ \Gammc{\func{r_{\G}^{\ast}}{\A^{A,\Phi},p^{A,\Phi}}} $ is Cauchy with respect to $ \Norm{\cdot}_{I}^{\GG,\lambda} $.

\item
For each $ a,b \in A_{0} $, there exists a (necessarily unique) $ m \in \M{A} $ having the following properties:
\begin{itemize}
\item
$ \Im{m}{A_{0}} \subseteq A_{0} $.

\item
For each $ \gamma \in \G $, we have $ \alphExtArg{\gamma}{m_{\func{s_{\G}}{\gamma}}} = m_{\func{r_{\G}}{\gamma}} $, where $ \overline{\alpha}_{\gamma} $ denotes the unique $ \ast $-isomorphism from $ \M{\Pair{A}{\Phi}_{\func{s_{\G}}{\gamma}}} $ to $ \M{\Pair{A}{\Phi}_{\func{r_{\G}}{\gamma}}} $ that extends $ \alpha_{\gamma} $.

\item
For each $ c \in A_{0} $ and $ u \in \Unit{\G} $, we have
$$
\FUNC{\func{m}{c}}{u} = \Int{\G^{u}}{\alphArg{\gamma}{\Func{a b^{\ast}}{\func{s_{\G}}{\gamma}}} \func{c}{\func{r_{\G}}{\gamma}}}{\lambda^{u}}{\gamma}.
$$
\end{itemize}
\end{enumerate}
Define $ \Inner{\cdot}{\cdot}{\PGDS{\GG}{\lambda}{A_{0}}}^{E}: A_{0} \times A_{0} \to \RCP{\GG}{\lambda} $ and $ \Inner{\cdot}{\cdot}{\PGDS{\GG}{\lambda}{A_{0}}}^{D}: A_{0} \times A_{0} \to \M{A} $ by
\begin{gather*}
\forall a,b \in A_{0}: \\
\begin{align*}
      \Inner{a}{b}{\PGDS{\GG}{\lambda}{A_{0}}}^{E}
& \df \lim_{i \in I}
      \func{\pi_{\r}^{\GG,\lambda}}
           {\Map{\G}{\A^{A,\Phi}}{\gamma}{\func{\varphi_{i}}{\gamma} \cdot \func{a}{\func{r_{\G}}{\gamma}}^{\ast} \alphArg{\gamma}{\func{b}{\func{s_{\G}}{\gamma}}}}}; \\
      \Inner{a}{b}{\PGDS{\GG}{\lambda}{A_{0}}}^{D}
& \df \text{The unique element} ~ m ~ \text{of} ~ \M{A} ~ \text{that satisfies (v)},
\end{align*}
\end{gather*}
where $ \Seq{\varphi_{i}}{i \in I} $ can be chosen to be any net in $ \Cc{\G,\CC{0}{1}} $ that converges uniformly to $ 1 $ on compact subsets of $ \G $, without risk of ambiguity. Let
\begin{align*}
\Eo{\GG}{\lambda}{A_{0}}  & \df \Span{\Set{\Inner{a}{b}{\PGDS{\GG}{\lambda}{A_{0}}}^{E}}{a,b \in A_{0}}} \subseteq \RCP{\GG}{\lambda}; \\
\E{\GG}{\lambda}{A_{0}}   & \df \Cl{\Eo{\GG}{\lambda}{A_{0}}}{\RCP{\GG}{\lambda}}; \\
\Do{\GG}{\lambda}{A_{0}}  & \df \Span{\Set{\Inner{a}{b}{\PGDS{\GG}{\lambda}{A_{0}}}^{D}}{a,b \in A_{0}}} \subseteq \M{A}; \\
\Fix{\GG}{\lambda}{A_{0}} & \df \Cl{\Do{\GG}{\lambda}{A_{0}}}{\M{A}}.
\end{align*}
Define a right $ \Eo{\GG}{\lambda}{A_{0}} $-action $ \circledast_{\PGDS{\GG}{\lambda}{A_{0}}}: A_{0} \times \Eo{\GG}{\lambda}{A_{0}} \to A_{0} $ on $ A_{0} $ by
\begin{gather*}
\forall a,a_{1},\ldots,a_{n},b_{1},\ldots,b_{n} \in A_{0}: \\
    a \circledast_{\PGDS{\GG}{\lambda}{A_{0}}} \sum_{j = 1}^{n} \Inner{a_{j}}{b_{j}}{\PGDS{\GG}{\lambda}{A_{0}}}^{E}
\df \sum_{j = 1}^{n} \func{\Inner{a}{a_{j}}{\PGDS{\GG}{\lambda}{A_{0}}}^{D}}{b_{j}}.
\end{gather*}
\end{Def}


In \cite{Rieffel2}, the definition of a proper $ C^{\ast} $-dynamical system is simpler as it does not involve Haar systems. Haar measures are unique up to positive scalings, and it can be shown that the resulting theory is independent of the choice of a Haar measure.

By the results of \cite{Brown} (mirroring those in \cite{Rieffel2}), the following statements hold:
\begin{itemize}
\item
$ \Eo{\GG}{\lambda}{A_{0}} $ is a $ \ast $-subalgebra of $ \RCP{\GG}{\lambda} $.

\item
$ \Do{\GG}{\lambda}{A_{0}} $ is a $ \ast $-subalgebra of $ \M{A} $.

\item
The triple
$$
\Xo{\GG}{\lambda}{A_{0}} \df \Trip{A_{0}}{\circledast_{\Br{\GG,\lambda;A_{0}}}}{\Inner{\cdot}{\cdot}{\PGDS{\GG}{\lambda}{A_{0}}}^{E}}
$$
gives us a pre-imprimitivity $ \Pair{\Do{\GG}{\lambda}{A_{0}}}{\Eo{\GG}{\lambda}{A_{0}}} $-bimodule.
\end{itemize}
We call $ \Fix{\GG}{\lambda}{A_{0}} $ a \emph{generalized fixed-point algebra}. By completing $ \Xo{\GG}{\lambda}{A_{0}} $, we obtain an imprimitivity $ \Pair{\Fix{\GG}{\lambda}{A_{0}}}{\E{\GG}{\lambda}{A_{0}}} $-bimodule $ \X{\GG}{\lambda}{A_{0}} $.


\section{A Bundle Structure for Generalized Fixed-Point Algebras}


In this section, we make the following standing assumptions:
\begin{enumerate}
\item \label{it1}
$ \GG = \Quad{\G}{A}{\Phi}{\alpha} $ is a separable groupoid dynamical system.

\item
$ \PGDS{\GG}{\lambda}{A_{0}} $ is a proper groupoid dynamical system.

\item
$ X $ is a locally compact Hausdorff space.

\item \label{it4}
$ q $ is a continuous map from $ \Unit{\G} $ to $ X $. \\
\textbf{Note:} This means that $ \Pair{A}{\Psi} $ is a $ \Co{X} $-algebra, with $ \Psi $ defined by
$$
\forall \varphi \in \Co{X}, ~ \forall a \in A: \quad
\FUNC{\func{\Psi}{\varphi}}{a} \df \FUNC{\func{\Phi^{\e}}{\varphi \circ q}}{a},
$$
where $ \Phi^{\e} $ denotes the $ \ast $-homomorphism from $ \Cb{X} $ to $ \Z{\M{A}} $ that extends $ \Phi $.

\item
$ \InvIm{q}{\SSet{x}} $ is $ \G $-invariant for each $ x \in X $, i.e., $ \func{q}{\func{s_{\G}}{\gamma}} = \func{q}{\func{r_{\G}}{\gamma}} $ for all $ \gamma \in \G $. \\
\textbf{Note:} For any open/closed subset $ Y $ of $ X $, adopt the following notation:
\begin{align*}
\G|_{Y}      & \df \Set{\gamma \in \G}{\func{q}{\func{r_{\G}}{\gamma}} \in Y}; \\
\lambda|_{Y} & \df \Seq{\lambda^{u}}{u \in \InvIm{q}{Y}}; \\
A|_{Y}       & \df \Set{\Seq{\func{a}{u}}{u \in \InvIm{q}{Y}}}{a \in A}; \\
\Phi|_{Y}    & \df \Map{\Co{\InvIm{q}{Y}}}{\Z{\M{A|_{Y}}}}{\varphi}{\text{Pointwise multiplication by} ~ \varphi}; \\
\alpha|_{Y}  & \df \Seq{\alpha_{\gamma}}{\gamma \in \G|_{Y}},
                   ~ \text{identifying} ~ \Pair{A|_{Y}}{\Phi|_{Y}}_{u} ~ \text{with} ~ \Pair{A}{\Phi}_{u} ~ \text{for} ~ u \in \InvIm{q}{\SSet{x}}.
\end{align*}
Then for each $ x \in X $, the given condition means that $ \G|_{\SSet{x}} $ is a closed sub-groupoid of $ \G $ and that $ \Unit{\Br{\G|_{\SSet{x}}}} = \InvIm{q}{\SSet{x}} $.

\item
$ \E{\GG}{\lambda}{A_{0}} $ is not just a $ \ast $-subalgebra but also an ideal of $ \RCP{\GG}{\lambda} $. This is guaranteed by an extra assumption in Rieffel's definition of proper $ C^{\ast} $-dynamical systems in \cite{Rieffel2} that is excluded from the definition of proper groupoid dynamical systems in \cite{Brown}, due to certain technical obstacles. However, \cite[Proposition A.1]{Brown/Goehle} does provide a checkable condition on a proper groupoid dynamical system ensuring that $ \E{\GG}{\lambda}{A_{0}} $ is an ideal.

\item
For each $ x \in X $, we have $ \RCP{\GG|_{\SSet{x}}}{\lambda|_{\SSet{x}}} = \FCP{\GG|_{\SSet{x}}}{\lambda|_{\SSet{x}}} $, where
$$
\GG|_{Y} \df \Quad{\G|_{Y}}{A|_{Y}}{\Phi|_{Y}}{\alpha|_{Y}}
$$
for any open/closed subset $ Y $ of $ X $.

\item \label{itlast}
$ A_{0} $ is closed under the left action $ \Psi $ of $ \Co{X} $ on $ A $.
\end{enumerate}

For now, fix $ x \in X $, and let
$$
A_{0}|_{\SSet{x}} \df \Set{\Seq{\func{a}{u}}{u \in \InvIm{q}{\SSet{x}}}}{a \in A_{0}}.
$$
Then $ A_{0}|_{\SSet{x}} $ is dense in $ A|_{\SSet{x}} $. Indeed, according to the Tietze Extension Theorem for upper-semicontinuous Banach bundles, we have
$$
A|_{\SSet{x}} = \Gammo{\Pair{\A^{A,\Phi}}{p^{A,\Phi}} \big|_{\InvIm{q}{\SSet{x}}}},
$$
and $ A_{0} $ is assumed to be already dense in $ A $.

Next, define a $ \ast $-homomorphism $ \Xi_{x}: \Fix{\GG}{\lambda}{A_{0}} \to \M{A|_{\SSet{x}}} $ by
\begin{gather*}
\forall m \in  \Fix{\GG}{\lambda}{A_{0}}, ~ \forall a \in A: \\
\FUNC{\func{\Xi_{x}}{m}}{\Seq{\func{a}{u}}{u \in \InvIm{q}{\SSet{x}}}} \df \Seq{\FUNC{\func{m}{a}}{u}}{u \in \InvIm{q}{\SSet{x}}}.
\end{gather*}
To see that $ \PGDS{\GG|_{\SSet{x}}}{\lambda|_{\SSet{x}}}{A_{0}|_{\SSet{x}}} $ is a proper groupoid dynamical system, note for each $ a,b \in A_{0}|_{\SSet{x}} $ that whenever $ a^{\sharp},b^{\sharp} \in A_{0} $ are chosen to satisfy
$$
\forall u \in \InvIm{q}{\SSet{x}}: \quad
\func{a^{\sharp}}{u} = \func{a}{u}
\quad \text{and} \quad
\func{b^{\sharp}}{u} = \func{b}{u},
$$
then the unique $ m \in \M{A|_{\SSet{x}}} $ satisfying (v) of \cref{Proper Groupoid Dynamical System} is $ \func{\Xi_{x}}{\Inner{a^{\sharp}}{b^{\sharp}}{\PGDS{\GG}{\lambda}{A_{0}}}^{D}} $.

Our goal now is to prove the following generalization of \cite[Theorem 3.2]{Rieffel2}.



\begin{Thm} \label{Main Theorem}
Under Assumptions \eqref{it1}-\eqref{itlast} above, the following are true:
\begin{enumerate}
\item \label{Main Proposition - Part 1}
$ \Pair{\Fix{\GG}{\lambda}{A_{0}}}{\Upsilon} $ is a $ \Co{X} $-algebra, where $ \Upsilon $ is uniquely characterized by
$$
  \FUNC{\func{\Upsilon}{\varphi}}{\sum_{j = 1}^{n} \Inner{a_{j}}{b_{j}}{\PGDS{\GG}{\lambda}{A_{0}}}^{D}}
= \sum_{j = 1}^{n} \Inner{\FUNC{\func{\Psi}{\overline{\varphi}}}{a_{j}}}{b_{j}}{\PGDS{\GG}{\lambda}{A_{0}}}^{D}
$$
for all $ \varphi \in \Co{X} $ and $ a_{1},\ldots,a_{n},b_{1},\ldots,b_{n} \in A_{0} $, and $ \Psi $ is defined as in Assumption \eqref{it4}.

\item \label{Main Proposition - Part 2}
$ \Pair{\Fix{\GG}{\lambda}{A_{0}}}{\Upsilon}_{x} \cong \Fix{\GG|_{\SSet{x}}}{\lambda|_{\SSet{x}}}{A_{0}|_{\SSet{x}}} $ for all $ x \in X $.
\end{enumerate}
\end{Thm}


Before proving this theorem, we need a few more lemmas.



\begin{Lem} \label{Fibering a Full Groupoid Crossed Product}
There exists a $ \Theta $ such that
\begin{enumerate}
\item
$ \Pair{\FCP{\GG}{\lambda}}{\Theta} $ is a $ \Co{X} $-algebra, and

\item
$ \Pair{\FCP{\GG}{\lambda}}{\Theta}_{x} \cong \FCP{\GG|_{\SSet{x}}}{\lambda|_{\SSet{x}}} $ for each $ x \in X $.
\end{enumerate}
\end{Lem}


\begin{proof}
Let $ x \in X $. By \cite[Theorem 5.22]{Goehle}, there is a short exact sequence
$$
                                         0
\longrightarrow                          \FCP{\GG|_{X \setminus \SSet{x}}}{\lambda|_{X \setminus \SSet{x}}}
\stackrel{\overline{i}}{\longrightarrow} \FCP{\GG}{\lambda}
\stackrel{\overline{r}}{\longrightarrow} \FCP{\GG|_{\SSet{x}}}{\lambda|_{\SSet{x}}}
\longrightarrow                          0.
$$
Here,
\begin{itemize}
\item
$ \overline{i} $ extends $ i: \Gammc{\func{r_{\G}^{\ast}}{\A^{A,\Phi},p_{A,\Phi}} \big|_{\G|_{X \setminus \SSet{x}}}} \to \Gammc{\func{r_{\G}^{\ast}}{\A^{A,\Phi},p_{A,\Phi}}} $, where
$$
\forall f \in \Gammc{\func{r_{\G}^{\ast}}{\A^{A,\Phi},p_{A,\Phi}} \big|_{\G|_{X \setminus \SSet{x}}}}: \quad
\func{i}{f} \df f \cup \Set{\Pair{\gamma}{0_{\Pair{A}{\Phi}_{\func{r_{\G}}{\gamma}}}}}{\gamma \in \G|_{\SSet{x}}},
$$

\item
$ \overline{r} $ extends $ r: \Gammc{\func{r_{\G}^{\ast}}{\A^{A,\Phi},p_{A,\Phi}}} \to \Gammc{\func{r_{\G}^{\ast}}{\A^{A,\Phi},p_{A,\Phi}} \big|_{\G|_{\SSet{x}}}} $, where
$$
\forall f \in \Gammc{\func{r_{\G}^{\ast}}{\A^{A,\Phi},p_{A,\Phi}}}: \quad
\func{r}{f} \df f|_{\G|_{\SSet{x}}}.
$$
\end{itemize}
Hence, the isometric image of $ \FCP{\GG|_{X \setminus \SSet{x}}}{\lambda|_{X \setminus \SSet{x}}} $ under $ \overline{i} $ in $ \FCP{\GG}{\lambda} $ is just
$$
\Cl{
   \Set{\func{\pi_{\u}^{\GG,\lambda}}{f}}
       {
       f \in \Gammc{\func{r_{\G}^{\ast}}{\A^{A,\Phi},p_{A,\Phi}}} ~ \text{and} ~
       \func{f}{\gamma} = 0_{\func{r_{\G}}{\gamma}} ~ \text{for all} ~ \gamma \in \G|_{\SSet{x}}
       }
   }{\FCP{\GG}{\lambda}}.
$$

As $ \GG $ is assumed separable, Renault's Disintegration Theorem says that $ \pi_{\u}^{\GG,\lambda} $ is unitarily equivalent to the integrated form of a covariant representation of $ \GG $. With this covariant representation, it can be seen for a fixed $ \varphi \in \Co{X} $ that
\begin{gather*}
\forall f \in \Gammc{\func{r_{\G}^{\ast}}{\A^{A,\Phi},p_{A,\Phi}}}: \\
     \Norm{\func{\pi_{\u}^{\GG,\lambda}}{\Br{\varphi \circ q \circ r_{\G}} \cdot f}}_{\FCP{\GG}{\lambda}}
\leq \Norm{\varphi}_{\infty} \Norm{\func{\pi_{\u}^{\GG,\lambda}}{f}}_{\FCP{\GG}{\lambda}}.
\end{gather*}
This yields a bounded linear operator $ \func{\Theta}{\varphi}: \FCP{\GG}{\lambda} \to \FCP{\GG}{\lambda} $ such that
\begin{gather*}
\forall \varphi \in \Co{X}, ~ \forall f \in \Gammc{\func{r_{\G}^{\ast}}{\A^{A,\Phi},p_{A,\Phi}}}: \\
\FUNC{\func{\Theta}{\varphi}}{\func{\pi_{\u}^{\GG,\lambda}}{f}} = \func{\pi_{\u}^{\GG,\lambda}}{\Br{\varphi \circ q \circ r_{\G}} \cdot f}.
\end{gather*}
Since $ \InvIm{q}{\SSet{x}} $ is invariant, $\func{\Theta}{\varphi} $ is a central multiplier on $ \FCP{\GG}{\lambda} $, and it can be proven that $ \Theta $ is a $ \ast $-homomorphism from $ \Co{X} $ to $ \Z{\M{\FCP{\GG}{\lambda}}} $.

To establish the non-degeneracy of $ \Theta $, let $ f \in \Gammc{\func{r_{\G}^{\ast}}{\A^{A,\Phi},p_{A,\Phi}}} $, $ \epsilon > 0 $, and
$$
K \df \Set{\gamma \in \G}{\Norm{\func{f}{\gamma}}_{\func{r_{\G}}{\gamma}} \geq \frac{\epsilon}{2 \Br{\Norm{f}_{I}^{\GG,\lambda} + 1}}},
$$
which is a compact subset of $ \G $. Pick $ \varphi \in \Cc{X} $ with range $ \CC{0}{1} $ such that $ \varphi \equiv 1 $ on $ \Im{q}{\Im{r_{\G}}{K}} $. Then $ \Norm{f - \Br{\varphi \circ q \circ r_{\G}} \cdot f}_{I}^{\GG,\lambda} < \epsilon $, which immediately yields
$$
\Norm{\func{\pi_{\u}^{\GG,\lambda}}{f} - \FUNC{\func{\Theta}{\varphi}}{\func{\pi_{\u}^{\GG,\lambda}}{f}}}_{\FCP{\GG}{\lambda}} < \epsilon.
$$
As $ f $ and $ \epsilon $ are arbitrary, the non-degeneracy of $ \Theta $ follows.

Finally, we must prove for any given $ x \in X $ that
$$
\Im{\overline{r}}{\FCP{\GG|_{X \setminus \SSet{x}}}{\lambda|_{X \setminus \SSet{x}}}} = \func{J}{\FCP{\GG}{\lambda},\Theta;x}.
$$
However, this follows as both are $ C^{\ast} $-subalgebras of $ \FCP{\GG}{\lambda} $ containing
$$
\Span{
     \Set{\FUNC{\func{\Theta}{\varphi}}{\func{\pi_{\u}^{\GG,\lambda}}{f}}}
         {\varphi \in I_{x} ~ \text{and} ~ f \in \Gammc{\func{r_{\G}^{\ast}}{\A^{A,\Phi},p_{A,\Phi}}}}
     }
$$
as a dense subset, where $ I_{x} \df \Set{\varphi \in \Co{X}}{\func{\varphi}{x} = 0} $. Therefore,
$$
      \Pair{\FCP{\GG}{\lambda}}{\Theta}_{x}
=     \Quot{\FCP{\GG}{\lambda}}{\func{J}{\FCP{\GG}{\lambda},\Theta;x}}
\cong \FCP{\GG|_{\SSet{x}}}{\lambda|_{\SSet{x}}}. \qedhere
$$
\end{proof}



\begin{Lem} \label{Reduced Equals Full}
Under the given assumption that
$$
\forall x \in X: \quad
\RCP{\GG|_{\SSet{x}}}{\lambda|_{\SSet{x}}} = \FCP{\GG|_{\SSet{x}}}{\lambda|_{\SSet{x}}},
$$
we have $ \RCP{\G}{\lambda} = \FCP{\G}{\lambda} $.
\end{Lem}


\begin{proof}
We follow the strategy of the second half of the proof of \cite[Theorem 3.5]{Rieffel1}. It suffices to see that irreducible representations of $ \FCP{\GG}{\lambda} $ factor through $ \RCP{\G}{\lambda} $. For once this is established, if the canonical quotient map from $ \FCP{\GG}{\lambda} $ onto $ \RCP{\GG}{\lambda} $ were not injective, then by picking a non-zero element $ b $ of $ \FCP{\GG}{\lambda} $ in the kernel of this quotient map and an irreducible representation $ \Pair{\sigma}{\mathcal{H}} $ of $ \FCP{\GG}{\lambda} $ such that
$$
     \Norm{\func{\sigma}{b}}_{\B{\mathcal{H}}}
=    \Norm{b}_{\FCP{\GG}{\lambda}}
\neq 0,
$$
but then $ \Pair{\sigma}{\mathcal{H}} $ cannot factor through $ \RCP{\GG}{\lambda} $ --- a contradiction.

It thus remains to prove that irreducible representations of $ \FCP{\GG}{\lambda} $ factor through $ \RCP{\G}{\lambda} $. Let $ \Pair{\sigma}{\mathcal{H}} $ be an irreducible representation of $ \FCP{\GG}{\lambda} $. We can uniquely extend $ \Pair{\sigma}{\mathcal{H}} $ to an irreducible representation $ \Pair{\tilde{\sigma}}{\mathcal{H}} $ of $ \M{\FCP{\GG}{\lambda}} $. By irreducibility,
$$
          \Im{\tilde{\sigma}}{\Z{\M{\FCP{\GG}{\lambda}}}}
\subseteq \Im{\tilde{\sigma}}{\M{\FCP{\GG}{\lambda}}}'
\subseteq \C \cdot \id_{\mathcal{H}},
$$
so the range of $ \tilde{\sigma} \circ \Theta $ is contained in $ \C \cdot \id_{\mathcal{H}} $, where $ \Theta $ is from \cref{Fibering a Full Groupoid Crossed Product}. Hence, there is an $ x \in X $ such that $ \Func{\tilde{\sigma} \circ \Theta}{\varphi} = \func{\varphi}{x} \cdot \id_{\mathcal{H}} $ for all $ \varphi \in \Co{X} $, so
\begin{align*}
\forall \varphi \in I_{x}, ~ \forall b \in \FCP{\GG}{\lambda}: \quad
    \func{\sigma}{\FUNC{\func{\Theta}{\varphi}}{b}}
& = \func{\tilde{\sigma}}{\func{\Theta}{\varphi}} \circ \func{\sigma}{b} \\
& = \SqBr{\func{\varphi}{x} \cdot \id_{\mathcal{H}}} \circ \func{\sigma}{b} \\
& = 0_{\B{\mathcal{H}}},
\end{align*}
which yields $ \func{J}{\FCP{\GG}{\lambda},\Theta;x} \subseteq \ker{\sigma} $. We thus have a representation $ \Pair{\dot{\sigma}}{\mathcal{H}} $ of
$$
\Quot{\FCP{\GG}{\lambda}}{\func{J}{\FCP{\GG}{\lambda},\Theta;x}} \cong \FCP{\GG|_{\SSet{x}}}{\lambda|_{\SSet{x}}}
$$
that satisfies
$$
\forall b \in \FCP{\GG}{\lambda}: \quad
\func{\dot{\sigma}}{b + \func{J}{\FCP{\GG}{\lambda},\Theta;x}} = \func{\sigma}{b}.
$$
However, $ \RCP{\GG|_{\SSet{x}}}{\lambda|_{\SSet{x}}} = \FCP{\GG|_{\SSet{x}}}{\lambda|_{\SSet{x}}} $; as $ \RCP{\GG|_{\SSet{x}}}{\lambda|_{\SSet{x}}} $ is known to be a quotient of $ \RCP{\GG}{\lambda} $, we get a representation $ \Pair{\tau}{\mathcal{H}} $ of $ \RCP{\GG}{\lambda} $ such that
$$
\forall f \in \Gammc{\func{r_{\G}^{\ast}}{\A^{A,\Phi},p_{A,\Phi}}}: \quad
\func{\tau}{\func{\pi_{\r}^{\GG,\lambda}}{f}} = \func{\sigma}{\func{\pi_{\u}^{\GG,\lambda}}{f}}.
$$
Therefore, $ \Pair{\sigma}{\mathcal{H}} $ factors through $ \RCP{\GG}{\lambda} $.
\end{proof}



\begin{Lem} \label{A Particular *-Isomorphism}
For each $ x \in X $, there is a $ \ast $-isomorphism
$$
\tilde{r}: \Quot{\E{\GG}{\lambda}{A_{0}}}{\func{J}{\E{\GG}{\lambda}{A_{0}},\Theta_{\E{\GG}{\lambda}{A_{0}}};x}} \to \E{\GG|_{\SSet{x}}}{\lambda|_{\SSet{x}}}{A_{0}|_{\SSet{x}}}
$$
such that
$$
\forall b \in \E{\GG}{\lambda}{A_{0}}: \quad
\func{\tilde{r}}{b + \func{J}{\E{\GG}{\lambda}{A_{0}},\Theta_{\E{\GG}{\lambda}{A_{0}}};x}} = \func{\overline{r}}{b},
$$
where $ \overline{r} $ denotes the surjective $ \ast $-homomorphism in the short exact sequence in the proof of \cref{Fibering a Full Groupoid Crossed Product}.
\end{Lem}


\begin{proof}
As $ C^{\ast} $-algebraic homomorphisms have closed ranges, the restriction of $ \overline{r}|_{\E{\GG}{\lambda}{A_{0}}} $ surjects onto $ \E{\GG|_{\SSet{x}}}{\lambda|_{\SSet{x}}}{A_{0}|_{\SSet{x}}} $, so there is a $ \ast $-isomorphism
$$
\tilde{r}:
\Quot{\E{\GG}{\lambda}{A_{0}}}{\ker{\overline{r}|_{\E{\GG}{\lambda}{A_{0}}}}} \to \E{\GG|_{\SSet{x}}}{\lambda|_{\SSet{x}}}{A_{0}|_{\SSet{x}}}
$$
such that
$$
\forall b \in \E{\GG}{\lambda}{A_{0}}: \quad
\func{\tilde{r}}{b + \ker{\overline{r}|_{\E{\GG}{\lambda}{A_{0}}}}} = \func{\overline{r}}{b}.
$$
It now remains to show that
$$
\ker{\overline{r}|_{\E{\GG}{\lambda}{A_{0}}}} = \func{J}{\E{\GG}{\lambda}{A_{0}},\Theta_{\E{\GG}{\lambda}{A_{0}}};x}.
$$
However,
$$
  \ker{\overline{r}|_{\E{\GG}{\lambda}{A_{0}}}}
= \E{\GG}{\lambda}{A_{0}} \cap \ker{\overline{r}}
= \E{\GG}{\lambda}{A_{0}} \cap \func{J}{\FCP{\GG}{\lambda},\Theta;x},
$$
so we already have
$$
\func{J}{\E{\GG}{\lambda}{A_{0}},\Theta_{\E{\GG}{\lambda}{A_{0}}};x} \subseteq \ker{\overline{r}|_{\E{\GG}{\lambda}{A_{0}}}}.
$$
To obtain the reverse inclusion, note that the action of $ I_{x} $ on $ \func{J}{\FCP{\GG}{\lambda},\Theta;x} $ via $ \Theta $ is non-degenerate.
\end{proof}


We are now ready to prove \cref{Main Theorem}.



\begin{proof}[{Proof of \cref{Main Theorem}}]
As $ \RCP{\GG}{\lambda} = \FCP{\GG}{\lambda} $ by \cref{Reduced Equals Full}, $ \E{\GG}{\lambda}{A_{0}} $ is an ideal of $ \FCP{\GG}{\lambda} $. Now, \cref{Fibering a Full Groupoid Crossed Product} says that $ \Pair{\FCP{\GG}{\lambda}}{\Theta} $ is a $ \Co{X} $-algebra, so $ \Pair{\E{\GG}{\lambda}{A_{0}}}{\Theta_{\E{\GG}{\lambda}{A_{0}}}} $ is also a $ \Co{X} $-algebra by \cref{An Ideal of a C0(X)-Algebra Is a C0(X)-Algebra}, but $ \Fix{\GG}{\lambda}{A_{0}} $ and $ \E{\GG}{\lambda}{A_{0}} $ are Morita equivalent, so \cref{Transferring C0(X)-Algebra Structures Across Imprimitivity Bimodules} gives us the desired outcome.

To finish Part \eqref{Main Proposition - Part 1}, it remains to establish the formula for $ \Upsilon $:
$$
  \FUNC{\func{\Upsilon}{\varphi}}{\sum_{j = 1}^{n} \Inner{a_{j}}{b_{j}}{\PGDS{\GG}{\lambda}{A_{0}}}^{D}}
= \sum_{j = 1}^{n} \Inner{\FUNC{\func{\Psi}{\overline{\varphi}}}{a_{j}}}{b_{j}}{\PGDS{\GG}{\lambda}{A_{0}}}^{D}.
$$
Throughout, let $ c_{1},\ldots,c_{n} \in A_{0} $ and $ f_{1},\ldots,f_{n} \in \Eo{\GG}{\lambda}{A_{0}} $. By \cref{Transferring C0(X)-Algebra Structures Across Imprimitivity Bimodules},
$$
  \FUNC{\func{\Upsilon}{\varphi}}{\sum_{j = 1}^{n} \Inner{a_{j}}{c_{j} \circledast_{\PGDS{\GG}{\lambda}{A_{0}}} f_{j}}{\PGDS{\GG}{\lambda}{A_{0}}}^{D}}
= \sum_{j = 1}^{n} \Inner{a_{j}}{c_{j} \circledast_{\PGDS{\GG}{\lambda}{A_{0}}} \FUNC{\func{\Theta}{\varphi}}{f_{j}}}{\PGDS{\GG}{\lambda}{A_{0}}}^{D}.
$$
It follows for all $ d \in A_{0} $ and $ u \in \Unit{\G} $ that
\begin{align*}
  & ~ \FUNC{\FUNC{\FUNC{\func{\Upsilon}{\varphi}}{\sum_{j = 1}^{n} \Inner{a_{j}}{c_{j} \circledast_{\PGDS{\GG}{\lambda}{A_{0}}} f_{j}}{\PGDS{\GG}{\lambda}{A_{0}}}^{D}}}{d}}
           {u} \\
= & ~ \FUNC{\sum_{j = 1}^{n} \func{\Inner{a_{j}}{c_{j} \circledast_{\PGDS{\GG}{\lambda}{A_{0}}} \FUNC{\func{\Theta}{\varphi}}{f_{j}}}{\PGDS{\GG}{\lambda}{A_{0}}}^{D}}{d}}
           {u} \\
= & ~ \sum_{j = 1}^{n}
      \Int{\G^{u}}
          {
          \alphArg{\gamma}{\FUNC{a_{j} \Br{c_{j} \circledast_{\PGDS{\GG}{\lambda}{A_{0}}} \FUNC{\func{\Theta}{\varphi}}{f_{j}}}^{\ast}}{\func{s_{\G}}{\gamma}}}
          \func{d}{\func{r_{\G}}{\gamma}}
          }
          {\lambda^{u}}
          {\gamma} \\
= & ~ \sum_{j = 1}^{n} \\
  & ~ \Int{\G^{u}}
          {
          \alphArg{\gamma}
                  {
                  \func{a_{j}}{\func{s_{\G}}{\gamma}}
                  \SqBr{
                       \func{\overline{\varphi}}{\func{q}{\func{s_{\G}}{\gamma}}} \cdot
                       \func{\Br{c_{j} \circledast_{\PGDS{\GG}{\lambda}{A_{0}}} f_{j}}^{\ast}}{\func{s_{\G}}{\gamma}}
                       }
                  }
          \func{d}{\func{r_{\G}}{\gamma}}
          }
          {\lambda^{u}}
          {\gamma} \\
= & ~ \sum_{j = 1}^{n}
      \Int{\G^{u}}
          {
          \alphArg{\gamma}
                  {
                  \func{\overline{\varphi}}{\func{q}{\func{s_{\G}}{\gamma}}} \cdot \func{a_{j}}{\func{s_{\G}}{\gamma}}
                  \func{\Br{c_{j} \circledast_{\PGDS{\GG}{\lambda}{A_{0}}} f_{j}}^{\ast}}{\func{s_{\G}}{\gamma}}
                  }
          \func{d}{\func{r_{\G}}{\gamma}}
          }
          {\lambda^{u}}
          {\gamma} \\
= & ~ \FUNC{
           \Func{\sum_{j = 1}^{n}
           \Inner{\FUNC{\func{\Psi}{\overline{\varphi}}}{a_{j}}}{c_{j} \circledast_{\PGDS{\GG}{\lambda}{A_{0}}} f_{j}}{\PGDS{\GG}{\lambda}{A_{0}}}^{D}}{d}
           }
           {u}.
\end{align*}
As $ \Span{A_{0} \circledast_{\PGDS{\GG}{\lambda}{A_{0}}} \Eo{\GG}{\lambda}{A_{0}}} $ is $ \Norm{\cdot}_{\Xo{\GG}{\lambda}{A_{0}}} $-dense in $ A_{0} $, the proof of Part \eqref{Main Proposition - Part 1} is complete

For Part \eqref{Main Proposition - Part 2}, fix $ x \in X $. We have seen that $ \PGDS{\GG|_{\SSet{x}}}{\lambda|_{\SSet{x}}}{A_{0}|_{\SSet{x}}} $ is a proper groupoid dynamical system. We must therefore show that there is a bounded linear operator
$$
T: \X{\GG}{\lambda}{A_{0}} \to \X{\GG|_{\SSet{x}}}{\lambda|_{\SSet{x}}}{A_{0}|_{\SSet{x}}}
$$
such that $ \func{T}{a} = a|_{\SSet{x}} $ for all $ a \in A_{0} $.

Recall the $ \ast $-homomorphism $ \overline{r} $ in the proof in \cref{Fibering a Full Groupoid Crossed Product}. Observe that
$$
\forall a,b \in A_{0}: \quad
  \func{\overline{r}}{\Inner{a}{b}{\PGDS{\GG}{\lambda}{A_{0}}}^{E}}
= \Inner{a|_{\SSet{x}}}{b|_{\SSet{x}}}{\PGDS{\GG|_{\SSet{x}}}{\lambda|_{\SSet{x}}}{A_{0}|_{\SSet{x}}}}^{E},
$$
which yields
$$
\forall a \in A_{0}: \qquad
\Norm{a|_{\SSet{x}}}_{\X{\GG|_{\SSet{x}}}{\lambda|_{\SSet{x}}}{A_{0}|_{\SSet{x}}}} \leq \Norm{a}_{\X{\GG}{\lambda}{A_{0}}}.
$$
As such, $ \Map{A_{0}}{A_{0}|_{\SSet{x}}}{a}{a|_{\SSet{x}}} $ extends to a bounded linear operator
$$
T: \X{\GG}{\lambda}{A_{0}} \to \X{\GG|_{\SSet{x}}}{\lambda|_{\SSet{x}}}{A_{0}|_{\SSet{x}}}.
$$

Now, notice that
$$
\forall a \in A_{0}, ~ \forall \varphi \in I_{x}, ~ \forall f \in \Eo{\GG}{\lambda}{A_{0}}: \quad
\func{T}{a \circledast_{\PGDS{\GG}{\lambda}{A_{0}}} \FUNC{\func{\Theta}{\varphi}}{f}} = 0_{A|_{\SSet{x}}}.
$$
Letting $ \X{\GG}{\lambda}{A_{0}}_{x} $ denote
$$\Cl{
     \Span{
          \X{\GG}{\lambda}{A_{0}} \mathbin{\overline{\circledast}}_{\PGDS{\GG}{\lambda}{A_{0}}}
          \func{J}{\E{\GG}{\lambda}{A_{0}},\Theta_{\E{\GG}{\lambda}{A_{0}}};x}
          }
     }{\X{\GG}{\lambda}{A_{0}}},
$$
where
$$
\overline{\circledast}_{\PGDS{\GG}{\lambda}{A_{0}}}: \X{\GG}{\lambda}{A_{0}} \times \E{\GG}{\lambda}{A_{0}} \to \X{\GG}{\lambda}{A_{0}}
$$
extends $ \circledast_{\PGDS{\GG}{\lambda}{A_{0}}} $. By continuity we have that
$$
\Im{T}{\X{\GG}{\lambda}{A_{0}}_{x}} = \SSet{0_{\X{\GG|_{\SSet{x}}}{\lambda|_{\SSet{x}}}{A_{0}|_{\SSet{x}}}}}.
$$
Therefore, there is  a bounded linear operator
$$
S: \Quot{\X{\GG}{\lambda}{A_{0}}}{\X{\GG}{\lambda}{A_{0}}_{x}} \to \X{\GG|_{\SSet{x}}}{\lambda|_{\SSet{x}}}{A_{0}|_{\SSet{x}}}
$$
such that
$$
\forall \zeta \in \X{\GG}{\lambda}{A_{0}}: \quad
\func{S}{\zeta + \X{\GG}{\lambda}{A_{0}}_{x}} = \func{T}{\zeta}.
$$
As $ \X{\GG}{\lambda}{A_{0}}_{x} $ is a Hilbert $ \func{J}{\E{\GG}{\lambda}{A_{0}},\Theta_{\E{\GG}{\lambda}{A_{0}}};x} $-module, the quotient $ \Quot{\X{\GG}{\lambda}{A_{0}}}{\X{\GG}{\lambda}{A_{0}}_{x}} $ is a Hilbert $ \Quot{\E{\GG}{\lambda}{A_{0}}}{\func{J}{\E{\GG}{\lambda}{A_{0}},\Theta_{\E{\GG}{\lambda}{A_{0}}};x}} $-module by \cite[Proposition 3.25]{Raeburn/Williams}. By \cref{A Particular *-Isomorphism}, we have the $ \ast $-isomorphism
$$
\tilde{r}:
\Quot{\E{\GG}{\lambda}{A_{0}}}{\func{J}{\E{\GG}{\lambda}{A_{0}},\Theta_{\E{\GG}{\lambda}{A_{0}}};x}} \to \E{\GG|_{\SSet{x}}}{\lambda|_{\SSet{x}}}{A_{0}|_{\SSet{x}}},
$$
so we can ask if $ S $ is a (unitary) isomorphism of Hilbert $ C^{\ast} $-modules.

Firstly, as
\begin{align*}
\forall a,b \in A_{0}: \quad
 {}& \func{\tilde{r}}{\Inner{a + \X{\GG}{\lambda}{A_{0}}_{x}}{b + \X{\GG}{\lambda}{A_{0}}_{x}}{\Quot{\X{\GG}{\lambda}{A_{0}}}{\X{\GG}{\lambda}{A_{0}}_{x}}}} \\
={}& \func{\tilde{r}}{\Inner{a}{b}{\PGDS{\GG}{\lambda}{A_{0}}}^{E} + \func{J}{\E{\GG}{\lambda}{A_{0}},\Theta_{\E{\GG}{\lambda}{A_{0}}};x}} \\
={}& \func{\overline{r}}{\Inner{a}{b}{\PGDS{\GG}{\lambda}{A_{0}}}^{E}} \\
={}& \Inner{a|_{\SSet{x}}}{b|_{\SSet{x}}}{\PGDS{\GG|_{\SSet{x}}}{\lambda|_{\SSet{x}}}{A_{0}|_{\SSet{x}}}}^{E} \\
={}& \Inner{a|_{\SSet{x}}}{b|_{\SSet{x}}}{\X{\GG|_{\SSet{x}}}{\lambda|_{\SSet{x}}}{A_{0}|_{\SSet{x}}}},
\end{align*}
we find that $ S $ is an isometry.  Secondly, $ S $ respects the $ C^{\ast} $-algebraic actions.

Lastly, the image of $ S $ is dense in $ \X{\GG|_{\SSet{x}}}{\lambda|_{\SSet{x}}}{A_{0}|_{\SSet{x}}} $, so as $ S $ is a Banach-space isometry, this image is all of $ \X{\GG|_{\SSet{x}}}{\lambda|_{\SSet{x}}}{A_{0}|_{\SSet{x}}} $. Therefore, $ S $ is indeed a (unitary) isomorphism of Hilbert $ C^{\ast} $-modules.

For any imprimitivity $ \Pair{B}{C} $-bimodule $ \Y $, it is known that the $ C^{\ast} $-algebra $ \K{\Y} $ of compact adjointable operators on $ \Y $ is $ \ast $-isomorphic to $ B $ (\cite[Proposition 3.8]{Raeburn/Williams}). Using this fact, we obtain
\begin{align*}
     {}& \Fix{\GG|_{\SSet{x}}}{\lambda|_{\SSet{x}}}{A_{0}|_{\SSet{x}}} \\
\cong{}& \K{\X{\GG|_{\SSet{x}}}{\lambda|_{\SSet{x}}}{A_{0}|_{\SSet{x}}}} \\
\cong{}& \K{\Quot{\X{\GG}{\lambda}{A_{0}}}{\X{\GG}{\lambda}{A_{0}}_{x}}} \qquad
          \Br{\text{As $ S $ is an isomorphism.}} \\
\cong{}& \Quot{\Fix{\GG}{\lambda}{A_{0}}}{\func{J}{\Fix{\GG}{\lambda}{A_{0}},\Upsilon;x}},
\end{align*}
where the last $ \ast $-isomorphism also follows from \cite[Proposition 3.25]{Raeburn/Williams}.
\end{proof}



\section{Examples}



\begin{Eg}
We will explain how to recover \cite[Theorem~3.2]{Rieffel2}. In line with the assumptions in this result, let $ G $ be a second-countable locally compact Hausdorff group, and $ X $ a second-countable locally compact Hausdorff space. Assume that $ \Pair{A}{\Phi} $ is a $ \Co{X} $-algebra and that $ \Seq{\alpha_{x}}{x \in X} $ is what Rieffel calls a \emph{continuous field of actions} of $ G $ on $ \Pair{A}{\Phi} $, i.e.,
\begin{itemize}
\item
$ \alpha_{x} $ is a $ \ast $-automorphism of $ \Pair{A}{\Phi}_{x} $ for each $ x \in X $, and

\item
$ \Map{X \times G}{\A^{A,\Phi}}{\Pair{x}{r}}{\alphArg{x,r}{\func{a}{x}}} $ is a continuous map for each $ a \in A $.
\end{itemize}
Assume a dense $ \ast $-subalgebra $ A_{0} $ of $ A $ exists such that $ \PCDS{G}{A|_{\SSet{x}}}{\alpha|_{\SSet{x}}}{A_{0}|_{\SSet{x}}} $ is a proper $ C^{\ast} $-dynamical system. Then consider the groupoid $ \G = X \times G $ with the following properties:
\begin{enumerate}
\item
$ \Unit{\G} = X \times \SSet{e} $.

\item
$ \func{s_{\G}}{x,r} = \Pair{x}{e} = \func{r_{\G}}{x,r} $ for all $ r \in G $ and $ x \in X $.

\item
The groupoid operations are defined by
$$
\forall r,s \in G, ~ \forall x \in X: \quad
\Pair{x}{r} \Pair{x}{s} \df \Pair{x}{r s}
\quad \text{and} \quad
\Pair{x}{r}^{- 1} = \Pair{x}{r^{- 1}}.
$$
\end{enumerate}
Define $ q: \Unit{\G} \to X $ by $ \func{q}{x,e} \df x $ for all $ x \in X $. Pick a Haar measure $ \mu $ on $ G $, and give $ \G $ the Haar system $ \lambda $ that imposes $ \mu $ on $ \G^{\Pair{x}{e}} = \SSet{x} \times G $ for each $ x \in X $. Then $ \Pair{A}{\Phi} $ satisfies our standing assumptions, is a $ \Co{\Unit{\G}} $-algebra, and
$$
\forall x \in X: \quad
\Pair{\Fix{\GG}{\lambda}{A_{0}}}{\Upsilon}_{x} = \Fix{\GG|_{\SSet{x}}}{\lambda|_{\SSet{x}}}{A_{0}|_{\SSet{x}}}.
$$
Let $ x \in X $, and observe the following:
\begin{itemize}
\item
$ \G|_{\SSet{x}} = \SSet{x} \times G \cong G $, so $ \G|_{\SSet{x}} $ reduces to a group.

\item
$ A|_{\SSet{x}} \cong \Pair{A}{\Phi}_{x} $.

\item
$ \alpha|_{\SSet{x}} \in \Aut{\Pair{A}{\Phi}_{x}} $.
\end{itemize}
One can check that
$$
          \Fix{\GG|_{\SSet{x}}}{\lambda|_{\SSet{x}}}{A_{0}|_{\SSet{x}}}
=         \FixR{\G|_{\SSet{x}}}{A|_{\SSet{x}}}{\alpha|_{\SSet{x}}}{A_{0}|_{\SSet{x}}}
\subseteq \M{A|_{\SSet{x}}}.
$$
Next, define an action $ \beta $ of $ G $ on $ A $ by identifying $ \func{\beta_{r}}{a} $ with $ \Seq{\alphArg{x,r}{\func{a}{x}}}{x \in X} $, for all $ r \in G $ and $ a \in A $. Then
$
\Fix{\GG}{\lambda}{A_{0}} = \FixR{G}{A}{\beta}{A_{0}}
$
because they are the same $ C^{\ast} $-subalgebra of $ \M{A} $ --- according to \cref{Proper Groupoid Dynamical System},
\begin{equation} \label{Example 1}
\forall a,b,c \in A_{0}: \quad
  \FUNC{\func{\Inner{a}{b}{\PGDS{\GG}{\lambda}{A_{0}}}^{D}}{c}}{x}
= \Int{G}{\alphArg{x,r}{\Func{a b^{\ast}}{x}} \func{c}{x}}{\mu}{r},
\end{equation}
while according to \cite{Rieffel2}, Rieffel's $ D $-inner product takes the form
$$
\forall a,b,c \in A_{0}: \quad
  \func{\Inner{a}{b}{\Br{G,A,\beta;A_{0}}}^{D,\operatorname{Rieffel}}}{c}
= \Int{G}{\func{\beta_{r}}{a b^{\ast}} c}{\mu}{r};
$$
however, the evaluation $ \ast $-homomorphism from $ A $ to $ \Pair{A}{\Phi}_{x} $ is continuous, so
\begin{align*}
\forall a,b,c \in A_{0}: \quad
    \FUNC{\func{\Inner{a}{b}{\Br{G,A,\beta;A_{0}}}^{D,\operatorname{Rieffel}}}{c}}{x}
& = \FUNC{\Int{G}{\func{\beta_{r}}{a b^{\ast}} c}{\mu}{r}}{x} \\
& = \Int{G}{\alphArg{x,r}{\Func{a b^{\ast}}{x}} \func{c}{x}}{\mu}{r},
\end{align*}
which is precisely the right-hand side of \eqref{Example 1}. Rieffel's result, then, is that
$$
\Pair{\FixR{G}{A}{\beta}{A_{0}}}{\Upsilon}_{x} \cong \FixR{\G|_{\SSet{x}}}{A|_{\SSet{x}}}{\alpha|_{\SSet{x}}}{A_{0}|_{\SSet{x}}},
$$
but this is just a consequence of our framework.
\end{Eg}



\begin{Eg} \label{Example 2}
Let $ \G $ be a second-countable and locally compact Hausdorff groupoid that acts freely and properly on $ \Unit{\G} $ (for the definition of a proper groupoid action, see \cite{Anantharaman-Delaroche/Renault}), and let $ \lambda $ be a Haar system on $ \G $. Fix the following objects:
\begin{enumerate}
\item
$ X \df \Orbit{\G}{\Unit{\G}} $ --- the orbit space of $ \Unit{\G} $ under the left action of $ \G $.

\item
$ q: \Unit{\G} \to X $ --- the cooresponding quotient map.

\item
$ \GG \df \Quad{\G}{\Co{\Unit{\G}}}{\Phi}{\lt} $, where
\begin{itemize}
\item
$ \Phi: \Co{\Unit{\G}} \to \M{\Co{\Unit{\G}}} $ is the left-multiplication action, and

\item
$ \lt $ denotes the left-translation action, i.e.,
$$
\func{\lt_{\gamma}}{f + \func{J}{\Co{\Unit{\G}},\Phi;\func{s_{\G}}{\gamma}}} \df g + \func{J}{\Co{\Unit{\G}},\Phi;\func{r_{\G}}{\gamma}}
$$
for any $ \gamma \in \G $ and $ f \in \Co{\Unit{\G}} $, where $ g \in \Co{\Unit{\G}} $ satisfies
$$
  \func{g}{\func{r_{\G}}{\gamma}}
= \func{f}{\gamma^{- 1} \cdot \func{r_{\G}}{\gamma}}
= \func{f}{\func{s_{\G}}{\gamma}}.
$$
As $ \Pair{\Co{\Unit{\G}}}{\Phi}_{u} \cong \SSet{u} \times \C $ for each $ u \in \Unit{\G} $, we can also define
$$
\forall k \in \C: \quad
\func{\lt_{\gamma}}{\func{s_{\G}}{\gamma},k} \df \Pair{\func{r_{\G}}{\gamma}}{k}.
$$
\end{itemize}
\end{enumerate}
By \cite[Proposition 4.1]{Brown}, $ \PGDS{\GG}{\lambda}{\Cc{\Unit{\G}}} $ is a proper groupoid dynamical system and $ \Fix{\GG}{\lambda}{\Cc{\Unit{\G}}} \cong \Co{\Orbit{\G}{\Unit{\G}}} $, so $ \Fix{\GG}{\lambda}{\Cc{\Unit{\G}}} $ can be made into a $ \Co{\Orbit{\G}{\Unit{\G}}} $-algebra.

To show that this is consistent with our framework, we will prove that
$$
\forall u \in \Unit{\G}: \quad
\Pair{\Fix{\GG}{\lambda}{\Cc{\Unit{\G}}}}{\Upsilon}_{\SqBr{u}_{\sim}} \cong \C.
$$
Fix $ u \in \Unit{\G} $, and let $ f,g,h \in \Cc{\Unit{\G}} \big|_{\SSet{\func{q}{u}}} = \Cc{\func{q}{u}} $. For all $ v \in \func{q}{u} $,
\begin{align*}
    \FUNC{\func{\Inner{f}{g}{\PGDS{\GG|_{\SSet{\func{q}{u}}}}{\lambda_{\SSet{\func{q}{u}}}}{\Cc{\func{q}{u}}}}^{D}}{h}}{v}
& = \Int{\G^{v}}{\func{\lt_{\gamma}}{\func{f \overline{g}}{\func{s_{\G}}{\gamma}}} \func{h}{\func{r_{\G}}{\gamma}}}{\lambda^{v}}{\gamma} \\
& = \Int{\G^{v}}{\func{f \overline{g}}{\func{s_{\G}}{\gamma}} \func{h}{\func{r_{\G}}{\gamma}}}{\lambda^{v}}{\gamma} \\
& = \SqBr{\Int{\G^{v}}{\func{f \overline{g}}{\func{s_{\G}}{\gamma}}}{\lambda^{v}}{\gamma}} \func{h}{v}.
\end{align*}
It suffices to prove that $ \D \Int{\G^{v}}{\func{f \overline{g}}{\func{s_{\G}}{\gamma}}}{\lambda^{v}}{\gamma} \in \C $ is a scalar independent of $ v $. Let $ v,w \in \func{q}{u} $, so that there exists an $ \eta \in \G_{v}^{w} $. Then
\begin{align*}
    \Int{\G^{v}}{\func{f \overline{g}}{\func{s_{\G}}{\gamma}}}{\lambda^{v}}{\gamma}
& = \Int{\G^{\func{r_{\G}}{\eta}}}{\func{f \overline{g}}{\func{s_{\G}}{\gamma}}}{\lambda^{\func{r_{\G}}{\eta}}}{\gamma} \\
& = \Int{\G^{\func{s_{\G}}{\eta}}}{\func{f \overline{g}}{\func{s_{\G}}{\eta \gamma}}}{\lambda^{\func{s_{\G}}{\eta}}}{\gamma} \quad
    \Br{\text{As $ \lambda $ is left-invariant.}} \\
& = \Int{\G^{\func{s_{\G}}{\eta}}}{\func{f \overline{g}}{\func{s_{\G}}{\gamma}}}{\lambda^{\func{s_{\G}}{\eta}}}{\gamma} \\
& = \Int{\G^{w}}{\func{f \overline{g}}{\func{s_{\G}}{\gamma}}}{\lambda^{w}}{\gamma}.
\end{align*}
Therefore, $ \Inner{f}{g}{\PGDS{\GG|_{\SSet{\func{q}{u}}}}{\lambda|_{\SSet{\func{q}{u}}}}{\Cc{\func{q}{u}}}}^{D} = C^{\func{q}{u}}_{f,g} \cdot 1_{\M{\Unit{\G}}} $ for some $ C^{\func{q}{u}}_{f,g} \in \C $, and with a little more work, it can be shown that $ \Map{\Orbit{\G}{\Unit{\G}}}{\C}{x}{C^{x}_{f,g}} \in \Co{\Orbit{\G}{\Unit{\G}}} $.
\end{Eg}



\begin{Eg}
As in \cref{Example 2}, let $ \G $ be a second-countable and locally compact Hausdorff groupoid that acts freely and properly on $ \Unit{\G} $, and let $ \lambda $ be a Haar system on $ \G $. Such groupoids arise in the study of groupoid equivalences (\cite{Brown/Goehle, Muhly/Raeburn/Williams}). Fix the following objects:
\begin{enumerate}
\item
$ X $ and $ q $ as in \cref{Example 2}.

\item
A separable groupoid dynamical system $ \GG = \Quad{\G}{A}{\Phi}{\alpha} $.

\item
If $ A_{0} \df \Span{\FUNC{\Im{\Phi}{\Cc{\Unit{\G}}}}{A}} $, then $ \PGDS{\GG}{\lambda}{A_{0}} $ is a proper groupoid dynamical system by \cite[Proposition~4.4]{Brown}.

\item
As $ \G $ is proper, $ \RCP{\GG|_{\SSet{x}}}{\lambda|_{\SSet{x}}} = \FCP{\GG|_{\SSet{x}}}{\lambda|_{\SSet{x}}} $ for all $ x \in X $ by \cite{Anantharaman-Delaroche/Renault}.
\end{enumerate}
This situation thus fits within our framework. By \cite[Proposition~3.6]{Brown/Goehle},
$$
\Fix{\GG}{\lambda}{A_{0}} =
\Set{f \in \Gammb{\A^{A,\Phi},p^{A,\Phi}}}
    {
    \begin{array}{c}
    \func{f}{\func{r_{\G}}{\gamma}} = \alphArg{\gamma}{\func{f}{\func{s_{\G}}{\gamma}}} ~ \text{and} \\
    \func{q}{u} \mapsto \Norm{f(u)} ~ \text{vanishes at} ~ \infty
    \end{array}
    }.
$$
Moreover, \cite[Proposition 3.6]{Brown/Goehle} shows that $ \Pair{\Fix{\GG}{\lambda}{A_{0}}}{\Upsilon} $ is a $ \Co{X} $-algebra whose fiber at $ x \in X $ is
\begin{align*}
      & ~ \Pair{\Fix{\GG}{\lambda}{A_{0}}}{\Upsilon}_{x} \\
=     & ~ \Set{f \in \Gammb{\Pair{\A^{A,\Phi}}{p^{A,\Phi}} \big|_{\InvIm{q}{\SSet{x}}}}}
              {
              \begin{array}{c}
              \func{f}{\func{r_{\G}}{\gamma}} = \alphArg{\gamma}{\func{f}{\func{s_{\G}}{\gamma}}} ~ \text{and} \\
              \func{q}{u} \mapsto \Norm{\func{f}{u}} ~ \text{vanishes at} ~ \infty
              \end{array}
              } \\
\cong & ~ \Fix{\GG|_{\SSet{x}}}{\lambda|_{\SSet{x}}}{A_{0}|_{\SSet{x}}},
\end{align*}
so our \cref{Main Theorem} recovers \cite[Proposition 3.4]{Brown/Goehle}. Note that
$$
\forall u \in \InvIm{q}{\SSet{x}}: \quad
\Fix{\GG|_{\SSet{x}}}{\lambda|_{\SSet{x}}}{A_{0}|_{\SSet{x}}} \cong \Pair{A}{\Phi}_{x}
$$
by \cite[Proposition 3.4]{Brown/Goehle}.
\end{Eg}


\section{Acknowledgements}
The main results of this paper were obtained while the first author was visiting the second author at the University of Colorado Boulder.  The first author would like to thank the second author for his hospitality during the visit.




\begin{thebibliography}{99}

\bibitem{Anantharaman-Delaroche/Renault}
\textsc{C. Anantharaman-Delaroche, J. Renault},
\textit{Amenable Groupoids},
Monogr. Enseign. Math., vol. 36,
L'Enseignement Math\'ematique,
Geneva 2000.

\bibitem{Brown}
\textsc{J. Brown},
Proper Actions of Groupoids on $ C^{\ast} $-Algebras,
\textit{J. Operator Theory},
\textbf{67}(2012), 437-467.

\bibitem{Brown/Goehle}
\textsc{J. Brown, G. Goehle},
The Brauer Semigroup of a Groupoid and a Symmetric Imprimitivity Theorem,
\textit{Trans. Amer. Math. Soc.},
\textbf{366}(2014), 1943-1972.

\bibitem{Goehle}
\textsc{G. Goehle},
\textit{Groupoid Crossed Products}, PhD Dissertation,
Dartmouth College, Hanover, 2009.

\bibitem{Khoshkam/Skandalis}
\textsc{M. Khoshkam, G. Skandalis},
Crossed Products of $ C^{\ast} $-Algebras by Groupoids and Inverse Semigroups,
\textit{J. Operator Theory},
\textbf{51}(2004), 255-279.

\bibitem{Muhly/Raeburn/Williams}
\textsc{P. Muhly, J. Renault, D. Williams},
Equivalence and Isomorphism for Groupoid $ C^{\ast} $-Algebras,
\textit{J. Operator Theory},
\textbf{17}(1987), 3-22.

\bibitem{Muhly/Williams}
\textsc{P. Muhly, D. Williams},
Renault's Equivalence Theorem for Groupoid Crossed Products,
\textit{NYJM Monographs},
\textbf{3}(2008), 1-87.

\bibitem{Raeburn/Williams}
\textsc{I. Raeburn, D. Williams},
\textit{Morita Equivalence and Continuous-Trace $ C^{\ast} $-Algebras},
Math. Surveys Monographs, vol. 60,
Amer. Math. Soc.,
Providence, RI (1998).

\bibitem{Rieffel1}
\textsc{M. Rieffel},
Continuous Fields of $ C^{\ast} $-Algebras Coming from Group Cocycles and Actions,
\textit{Math. Ann.},
\textbf{283}(1989), 631-643.

\bibitem{Rieffel2}
\textsc{M. Rieffel},
Proper Actions of Groups on $ C^{\ast} $-Algebras,
\textit{Mappings of Operator Algebras},
Progr. Math., vol. 84,
Birkh\"auser,
Boston, MA (1990), 141-182.

\bibitem{Williams}
\textsc{D. Williams},
\textit{Crossed Products of $ C^{\ast} $-Algebras},
Math. Surveys Monographs, vol. 134,
Amer. Math. Soc.,
Providence, RI (2007).

\end{thebibliography}
\end{document}